\pdfoutput=1

\documentclass[10pt]{article}

\makeatletter
\renewcommand{\maketitle}{\bgroup\setlength{\parindent}{0pt}
\begin{flushleft}
  \textbf{\@title}

  \@author
\end{flushleft}\egroup
} 
\def\footnoterule{\kern-3\p@
  \hrule \@width 2in \kern 2.6\p@} 
\makeatother

\usepackage{preamble}
\usepackage{courier}
\usepackage[square,numbers]{natbib}

\title{{\Large \textbf{Forking and invariant measures in NIP theories}}}
\author{Anand Pillay,\footnote[1]{Partially supported by NSF grants DMS-1760212 and DMS-2054271.} Atticus Stonestrom\\{\small Department of Mathematics, University of Notre Dame}}

\begin{document}
\maketitle

{\small\noindent\textbf{Abstract:} We give an example of an NIP theory $T$ in which there is a formula that does not fork over $\varnothing$ but has measure $0$ under any global $\varnothing$-invariant Keisler measure, and we show that this cannot occur if $T$ is also first-order amenable. \newline

\noindent\textbf{Notation:} Throughout $T$ will denote a complete theory in a first-order language $L$, with a monster model $\mathfrak{C}$. We will always be working over $\varnothing$. So: (i) if we say a formula or type forks (respectively divides), we mean that it forks (respectively divides) over $\varnothing$, (ii) by `definable' (respectively `type-definable') we mean definable (respectively type-definable) without parameters, and (iii) if we say two tuples $a,b$ are conjugate, we mean they have the same type over $\varnothing$, which we denote by $a\equiv b$.}

\section{Introduction}
This paper continues investigation of one of the themes of \citep{seven_authors}, namely the relationship between two natural ideals of $L(\mathfrak{C})$-formulas in a theory $T$: the ideal of formulas which fork, and the ideal of formulas which are given measure $0$ under any automorphism-invariant Keisler measure on $\mathfrak{C}$. The former is always contained in the latter (see for example Fact 1.2 in \citep{seven_authors}), and the other inclusion holds if $T$ is stable (see for example Corollary 1.4 in \citep{seven_authors}). But the other inclusion need not hold in general, and \citep{seven_authors} gives such an example where $T$ is simple. In this paper we will discuss the case when $T$ is NIP; for much more general context, we refer to the introduction of \citep{seven_authors}.

From here on, let $\phi(x,b)$ be an $L(\mathfrak{C})$-formula that does not fork, and let $(\ast)$ denote the statement that $\mu(\phi(x,b))>0$ for some global automorphism-invariant Keisler measure $\mu$. It is an early fact from \citep{adler_nip}, that, if $T$ is NIP, then $\phi(x,b)$ is contained in a global type which is Lascar-invariant over $\varnothing$; in particular, if $\dcl(\varnothing)$ is a model, then $\phi(x,b)$ is contained in some automorphism-invariant \textit{type}, and so $(\ast)$ holds. In \citep{hrushovski_pillay}, Hrushovski and the first author showed further that, if $T$ is NIP, then $\phi(x,b)$ is contained in a global type which is KP-invariant over $\varnothing$; by lifting the Haar measure on the KP-Galois group of $T$, they used this to show that, if $T$ is NIP, then $(\ast)$ holds provided $b\in\dcl(\varnothing)$ (see Section 2.1 below for more details). Nevertheless, we show here that $(\ast)$ need not hold in general when $T$ is NIP; this is the content of Section 4.

A notion related to the above theme is that of \textit{first-order amenability}, introduced by Hrushovski, Krupiński, and the first author in \citep{hrushovski_krupinski_pillay_1} and \citep{hrushovski_krupinski_pillay_2}. $T$ is said to be `amenable' if every type over $\varnothing$ extends to a global automorphism-invariant Keisler measure. Using the construction from \citep{hrushovski_pillay} mentioned above, \citep{hrushovski_krupinski_pillay_1} shows that, if $T$ is NIP, then $T$ is amenable if and only if $\varnothing$ is an extension base – ie no formula without parameters forks. In Section 3 we observe that $(\ast)$ holds when $T$ is NIP and amenable; in fact, we give two proofs, both adapting arguments of Chernikov and Simon. In contrast, as shown in \citep{seven_authors}, $(\ast)$ need not hold even if $T$ is simple and amenable, as the example found there \textit{is} amenable.

\section{Preliminaries}
\subsection{KP-Strong Types}
Recall that, for a small tuple of variables $x$, two $x$-tuples $b,c\in\mathfrak{C}^x$ are said to be \textit{Kim-Pillay equivalent} (over $\varnothing$), written $E_{\mathrm{KP}}(b,c)$, if they lie in the same equivalence class of every bounded type-definable equivalence relation on $\mathfrak{C}^x$. For small tuples of variables $x,y$ of the same arities and respective sorts, the relation $E_{\mathrm{KP}}(x,y)$ is bounded and type-definable, and the KP-strong type of $b$ is the equivalence class of $b$ under $E_{\mathrm{KP}}$. The KP-Galois group, denoted $\mathrm{Gal}_{\mathrm{KP}}(T)$, is the quotient of $\Aut(\mathfrak{C})$ by the (normal) subgroup of automorphisms $\sigma$ such that $E_{\mathrm{KP}}(b,\sigma(b))$ holds for every small tuple $b$. Given an automorphism $\sigma$ of $\mathfrak{C}$, we will denote by $\bar{\sigma}$ its image in $\mathrm{Gal}_{\mathrm{KP}}(T)$.

The type of a tuple over any small model determines its KP-strong type. Thus, for any small $M,N\prec\mathfrak{C}$, if we let $\bar{n}$ be an enumeration of $N$ and let $S_{\bar{n}}(M)$ denote the space of complete types over $M$ extending $\tp(\bar{n}/\varnothing)$, then the map taking $\tp(\sigma(\bar{n})/M)$ to $\bar{\sigma}$ gives a well-defined surjection from $S_{\bar{n}}(M)$ to $\mathrm{Gal}_{\mathrm{KP}}(T)$. The quotient topology on $\mathrm{Gal}_{\mathrm{KP}}(T)$ induced by this map is independent of the choice of $M,N$, and makes $\mathrm{Gal}_{\mathrm{KP}}(T)$ into a compact Hausdorff topological group; see for example \citep{ziegler}.

In particular, $\mathrm{Gal}_{\mathrm{KP}}(T)$ is equipped with a bi-invariant Haar measure of weight $1$, ie a regular Borel measure $\eta$ such that $\eta(\mathrm{Gal}_{\mathrm{KP}}(T))=1$ and such that $\eta(\bar{\sigma}S\bar{\tau})=\eta(S)$ for any Borel subset $S\subseteq\mathrm{Gal}_{\mathrm{KP}}(T)$ and any $\bar{\sigma},\bar{\tau}\in \mathrm{Gal}_{\mathrm{KP}}(T)$.

\subsection{Forking in NIP Theories}
In this section assume that $T$ is NIP. Recall that a global type $p(x)\in S_x(\mathfrak{C})$ is `KP-invariant' if, for any $L(\mathfrak{C})$-formula $\phi(x,b)$ and any $b'$ with the same KP-strong type as $b$, we have $p(x)\vdash \phi(x,b)\leftrightarrow\phi(x,b')$. From \citep{hrushovski_pillay} we have the following:
\begin{fact}
    A global type does not fork if and only if it is KP-invariant.
\end{fact}

In particular, given a global non-forking type $p(x)\in S_x(\mathfrak{C})$ and an $L(\mathfrak{C})$-formula $\phi(x,b)$, the subset $S_{p,\phi(x,b)}\subseteq\mathrm{Gal}_{\mathrm{KP}}(T)$ given by $\{\bar{\sigma}:\sigma p(x)\vdash \phi(x,b)\}$ is well-defined. Again from \citep{hrushovski_pillay}, we have:

\begin{fact}
    If $p(x)\in S_x(\mathfrak{C})$ does not fork, then $S_{p,\phi(x,b)}$ is Borel for every $L(\mathfrak{C})$-formula $\phi(x,b)$. In particular, taking $\mu_p(\phi(x,b)):=\eta(S_{p,\phi(x,b)})$ for each $L(\mathfrak{C})$-formula $\phi(x,b)$ gives a well-defined Keisler measure on $\mathfrak{C}^x$, which is automorphism-invariant by the left-invariance of $\eta$.
\end{fact}

Now, recall that a small subset $B\subset\mathfrak{C}$ is called an `extension base' if no type over $B$ forks over $B$. We have the following from \citep{chernikov_kaplan}:
\begin{fact}
    Suppose $B$ is an extension base. Then an $L(\mathfrak{C})$-formula forks over $B$ if and only if it divides over $B$. Moreover, given an $L$-formula $\psi(x,y)$, the set of all $c\in\mathfrak{C}^y$ such that $\psi(x,c)$ forks over $B$ is type-definable over $B$.
\end{fact}

We will also need the following, which is Proposition 25 of \citep{chernikov_simon_pq}; it is stated and proved there for the case when $\dcl(\varnothing)$ is a model, but precisely the same proof works over any extension base, so we repeat the proof here for clarity without the assumption that $\dcl(\varnothing)$ is a model. For the statement of the $(p,q)$-property and the $(p,q)$-theorem see Section 0.4 in \citep{chernikov_simon_pq}.

\begin{fact}
Suppose $\varnothing$ is an extension base, that $\phi(x,y)$ is an $L$-formula, and that $\pi(y)$ is a partial type over $\varnothing$ such that $\phi(x,b)$ does not fork for all $b\models\pi$. Then there are finitely many non-forking global types $p_1,\dots,p_n\in S_x(\mathfrak{C})$ such that, for every $b\models\pi$, there is some $p_i$ concentrated on $\phi(x,b)$.
\end{fact}
\begin{proof} Let $X$ denote the closed subset of $S_x(\mathfrak{C})$ consisting of those types which do not fork. Give $b\models\pi$, let $F_b=\{p\in X:p\vdash\phi(x,b)\}$, and let $\mathcal{F}$ be the family $\{F_p:b\models\pi\}$. Since $T$ is NIP, $\mathcal{F}$ has finite VC-dimension, and hence finite VC-codimension – say strictly smaller than $q\in\omega$. We claim that $\mathcal{F}$ has the $(p,q)$-property for some $p\in\omega$. Otherwise, for each $p\in\omega$, there is a sequence $b_1,\dots,b_p$ of realizations of $\pi$ such that, for any $s\subseteq[p]$ of size $q$, we have $\bigcap_{i\in s}F_{b_i}=\varnothing$, so that the formula $\bigwedge_{i\in s}\phi(x,b_i)$ forks. 

Now, since $\varnothing$ is an extension base, the set of $(c_1,\dots,c_q)$ such that $\bigwedge_{i\in[q]}\phi(x,c_i)$ forks is type-definable. Thus, by compactness, we may find an infinite sequence $(b_i)_{i\in\omega}$ of realizations of $\pi(y)$ such that $\bigwedge_{i\in s}\phi(x,b_i)$ forks for every $s\subset\omega$ of size $q$, and then again by compactness and by Ramsey's theorem we may assume $(b_i)_{i\in\omega}$ is indiscernible. But this is a contradiction; letting $q(x)$ be any global non-forking type concentrated on $\phi(x,b_0)$, then by Lascar-invariance $q$ is concentrated on $\phi(x,b_i)$ for each $i\in\omega$, so that $\bigwedge_{i\in[q]}\phi(x,b_i)$ does not fork.

So indeed $\mathcal{F}$ has the $(p,q)$-property for some $p\in\omega$. By the $(p,q)$-theorem, there is thus some $n\in\omega$ such that, for every finite subset $\mathcal{F}_0\subseteq\mathcal{F}$, there is an $n$-element subset of $X$ which meets every element of $\mathcal{F}_0$. Hence, by compactness, the partial type in variables $x_1,\dots,x_n$ containing the following formulas is consistent: (i) $\bigvee_{i\in[n]}\phi(x_i,b)$ for every $b\models \pi$, and (ii) $\neg\psi(x_i,c)$ for every $i\in[n]$ and every $\psi(x_i,c)\in L(\mathfrak{C})$ which forks. Taking any global completion and letting $p_i$ be its restriction to the variable $x_i$, the desired result follows.
\end{proof}

\subsection{Dynamics}
Let $G$ be a (discrete) group. Recall that a `$G$-flow' is a non-empty compact Hausdorff space $X$ together with a left action of $G$ on $X$ by homeomorphisms. A `subflow' of $X$ is a closed $G$-invariant subset of $X$, ie a non-empty closed subset $Y\subseteq X$ with $gY=Y$ for all $g\in G$. The flow $X$ is `point-transitive' if, for some $p\in X$, the orbit $Gp$ is a dense subset of $X$. An open subset $U\subseteq X$ is called `generic' if $X=g_1U\cup\dots\cup g_nU$ for some $g_1,\dots,g_n\in G$, and is called `weakly generic' if there is a non-generic open set $V$ such that $U\cup V$ is generic. A point $p\in X$ is called (weakly) generic if every one of its open neighborhoods is (weakly) generic. Finally, a point $p\in X$ is said to be `almost periodic' if it lies in a minimal subflow of $X$. The following is proved in \citep{newelski}:
\begin{fact}
    Let $X$ be a point-transitive $G$-flow. Then every almost periodic point of $X$ is weakly generic.
\end{fact}

\section{Amenable Case}
Recall from \citep{hrushovski_krupinski_pillay_1} that a theory $T$ is said to be `first-order amenable', or just `amenable', if every type over $\varnothing$ extends to a global automorphism-invariant Keisler measure. In any theory $T$, amenability implies that $\varnothing$ is an extension base, and, using Fact 2.2, it is observed in \citep{hrushovski_krupinski_pillay_1} that the converse holds when $T$ is NIP. We now observe that, if $T$ is NIP and first-order amenable, then having measure $0$ under any automorphism-invariant Keisler measure implies forking. This is the `automorphism group analogue' of Corollary 3.34 in \citep{chernikov_simon_groups} and is proved in an analogous way, with the proof a quick corollary of Fact 2.2 and Fact 2.4.
\begin{theorem}
Suppose that $T$ is NIP, that $\varnothing$ is an extension base, and that $\phi(x,b)$ is a non-forking $L(\mathfrak{C})$-formula. Then there is a global automorphism-invariant Keisler measure giving $\phi(x,b)$ positive measure.
\end{theorem}
\begin{proof}
    By Fact 2.4 applied to the type $\pi(x):=\tp(b/\varnothing)$, there are non-forking types $p_1,\dots,p_n\in S_x(\mathfrak{C})$ such that, for every $b'\equiv b$, the formula $\phi(x,b')$ is contained in some $p_i$. In particular, the sets $S_{p_1,\phi(x,b)},\dots,S_{p_n,\phi(x,b)}$ defined in Section 2.2 cover $\mathrm{Gal}_{\mathrm{KP}}(T)$. Hence, for some $i\in[n]$, $S_{p_i,\phi(x,b)}$ has Haar measure at least $1/n$, and then the measure $\mu_{p_i}$ from Fact 2.2 is automorphism-invariant and gives $\phi(x,b)$ measure at least $1/n$.
\end{proof}

\subsection{Connection with Dynamics}
Corollary 3.34 in \citep{chernikov_simon_groups} is deduced using a `definable group analogue' of Fact 2.4,\footnote[1]{The analogue in this case is that, in a definably amenable NIP group, for any f-generic formula $\phi(x,b)$ there are finitely many global f-generic types $p_1,\dots,p_n$ such that every translate of $\phi(x,b)$ is contained in some $p_i$.} which is proved using topological dynamics in Section 3.3 of \citep{chernikov_simon_groups}. We observe in this section that an analogous argument also gives a topological dynamics proof of Fact 2.4.

Let $x$ be a tuple of variables. Then $S_x(\mathfrak{C})$ is naturally a flow for the automorphism group $G=\mathrm{Aut}(\mathfrak{C})$, and a subset $Y\subseteq S_x(\mathfrak{C})$ is a subflow if and only if it is of form $S_{\pi(x)}(\mathfrak{C})$, where $\pi(x)$ is an automorphism-invariant partial type over $\mathfrak{C}$ and $S_{\pi(x)}(\mathfrak{C})$ is the set of complete $x$-types over $\mathfrak{C}$ extending $\pi(x)$. Fix such a $\varnothing$-invariant partial type $\pi(x)$.

Let us say that an $L(\mathfrak{C})$-formula $\phi(x,b)$ is (weakly) generic for $\pi(x)$ if the clopen subset of $S_{\pi(x)}(\mathfrak{C})$ corresponding to $\phi(x,b)$ is (weakly) generic; then $\phi(x,b)$ is generic for $\pi(x)$ if and only if $\pi(x)\vdash \bigvee_{i\in[n]}\phi(x,b_i)$ for some finitely many $b_1,\dots,b_n$ with $b_i\equiv b$ for each $i\in[n]$, and, by compactness in $S_{\pi(x)}(\mathfrak{C})$, $\phi(x,b)$ is weakly generic for $\pi(x)$ if and only if there is an $L(\mathfrak{C})$-formula $\psi(x,c)$ not generic for $\pi(x)$ such that $\phi(x,b)\vee\psi(x,c)$ is generic for $\pi(x)$.

Now the following lemmas correspond respectively to Proposition 3.33 and Proposition 3.30 in \citep{chernikov_simon_groups}:

\begin{lemma}
    Suppose that $T$ is NIP and that $\phi(x,b)$ is a non-dividing $L(\mathfrak{C})$-formula. Then there are finitely many almost periodic types $p_1,\dots,p_n\in X=S_{x}(\mathfrak{C})$ such that, for any $b'\equiv b$, there is $i\in[n]$ with $p_i\vdash\phi(x,b')$.
\end{lemma}
\begin{proof}
    Given $b'\equiv b$, let $F_{b'}=\{p\in X:p\vdash\phi(x,b')\}$. Since $T$ is NIP, the family of subsets $\mathcal{F}=\{F_{b'}:b'\equiv b\}$ has finite VC-codimension, say strictly smaller than $q\in\omega$. Now $\mathcal{F}$ has the $(p,q)$-property for some $p\in\omega$; otherwise we can find arbitrary long sequences $b_1,\dots,b_p$ of conjugates of $b$ such that, for any $s\subseteq[p]$ of size $q$, we have $\bigcap_{i\in s}F_{b_i}=\varnothing$, so that $\bigwedge_{i\in s}\phi(x,b_i)$ is inconsistent. But then by Ramsey and compactness we may find an indiscernible sequence $(b_i)_{i\in\omega}$ of conjugates of $b$ such that $\bigwedge_{i\in s}\phi(x,b_i)$ is inconsistent for every $s\subset\omega$ of size $q$, contradicting that $\phi(x,b)$ does not divide. So indeed $\mathcal{F}$ has the $(p,q)$-property for some $p\in\omega$, and so by compactness and the $(p,q)$-theorem we may find finitely many $p_1,\dots,p_n\in X$ such that, for every $b'\equiv b$, there is some $i\in[n]$ with $p_i\vdash\phi(x,b')$.
    
    Now consider the flow $X^n$, equipped with the product topology and on which $G$ acts diagonally. Let $Y$ be the closure in $X^n$ of the orbit of $(p_1,\dots,p_n)$. We claim that, for any $(q_1,\dots,q_n)\in Y$, and any $b'\equiv b$, there is some $i\in[n]$ with $q_i\vdash\phi(x,b')$. To see this, suppose otherwise. Let $O$ be the clopen subset of $X$ corresponding to the formula $\neg\phi(x,b')$; then $O^n$ is an open subset of $Y$ containing $(q_1,\dots,q_n)$. But $O^n$ is also disjoint from the orbit of $(p_1,\dots,p_n)$; indeed, if $\sigma$ is an automorphism, then some $p_i$ contains $\phi(x,\sigma^{-1}(b'))$, and so some $\sigma p_i$ contains $\phi(x,b')$. So $(q_1,\dots,q_n)$ lying in $O^n$ contradicts the definition of $Y$.
    
    By Zorn's lemma, let $Z$ be a minimal subflow of $Y$. Then the image of $Z$ under any of the projection maps $X^n\to X$ is a minimal subflow of $X$. In particular, picking any $(q_1,\dots,q_n)\in Z$, then the types $q_1,\dots,q_n$ have the desired properties.
\end{proof}

\begin{lemma}
    Suppose that forking and dividing coincide, and that $\pi(x)$ is an automorphism-invariant partial type over $\mathfrak{C}$. Then a weakly generic formula for $\pi(x)$ does not fork.
\end{lemma}
\begin{proof}
    Suppose that $\phi(x,b)$ forks and that $\phi(x,b)\vee\psi(x,c)$ is generic for $\pi(x)$; we will show that $\psi(x,c)$ is generic for $\pi(x)$, which will give the desired result. Since $\phi(x,b)\vee\psi(x,c)$ is generic for $\pi(x)$, there are $(b_1,c_1),\dots,(b_n,c_n)$ with $(b_i,c_i)\equiv (b,c)$ for each $i\in[n]$ and $\pi(x)\vdash\bigvee_{i\in[n]}(\phi(x,b_i)\vee\psi(x,c_i))$. In other words, we have $\pi(x)\vdash \zeta(x,\bar{b})\vee\xi(x,\bar{c})$, where \begin{align*}
    \zeta(x,y_1,\dots,y_n)&\equiv\bigvee_{i\in[n]}\phi(x,y_i)\\
    \xi(x,z_1,\dots,z_n)&\equiv\bigvee_{i\in[n]}\psi(x,z_i)
    \end{align*} and $\bar{b}=(b_1,\dots,b_n),\bar{c}=(c_1,\dots,c_n)$. Now, the formulas which fork form an ideal, and $\phi(x,b_i)$ forks for each $i$, so $\zeta(x,\bar{b})$ forks as well. Since forking and dividing coincide, there is thus an indiscernible sequence $(\bar{b}_k=(b_{k1},\dots,b_{kn}))_{k\in\omega}$ such that $\bar{b}_0=\bar{b}$ and such that the family of formulas $(\zeta(x,\bar{b}_k):k\in\omega)$ is inconsistent.

    Since $\bar{b}\equiv\bar{b}_k$ for each $k$, we may find $(\bar{c}_k:k\in\omega)$ such that $(\bar{b}_k,\bar{c}_k)\equiv (\bar{b},\bar{c})$ for each $k$. In particular, $\pi(x)\vdash\zeta(x,\bar{b}_k)\vee\xi(x,\bar{c}_k)$ for each $k$. On the other hand, for some $q\in\omega$, we have $\bigwedge_{k\in[q]}\zeta(x,\bar{b}_k)$ inconsistent, whence $\pi(x)\vdash\bigvee_{k\in[q]}\xi(x,\bar{c}_k)$, ie $\pi(x)\vdash\bigvee_{k\in[q]}\bigvee_{i\in[n]}\psi(x,c_{ki}).$ Since $\bar{c}\equiv\bar{c}_k$ for each $k\in\omega$, it follows that $c_{ki}\equiv c$ for each $k,i\in\omega$, and hence indeed $\psi(x,c)$ is generic for $\pi(x)$.
\end{proof}

Now Fact 2.4 follows from Lemma 3.2 and Lemma 3.3; indeed, suppose $T$ is NIP and that $\varnothing$ is an extension base, and let $\phi(x,b)$ be any non-forking formula. By Lemma 3.2, there are almost periodic types $p_1,\dots,p_n\in S_{x}(\mathfrak{C})$ such that any conjugate of $\phi(x,b)$ lies in some $p_i$, so it suffices to show that $p_i$ is non-forking for each $i\in[n]$.

To see this, let $Y_i$ be the (unique) minimal subflow of $S_x(\mathfrak{C})$ containing $p_i$, and let $\pi_i(x)$ be an automorphism-invariant partial type over $\mathfrak{C}$ such that $Y_i=S_{\pi_i(x)}(\mathfrak{C})$. $Y_i$ is a minimal flow, so by Fact 2.5 any point of $Y_i$ is weakly generic in $Y_i$, and in particular $p_i$ is. But $\varnothing$ is an extension base, so forking and dividing coincide by Fact 2.3, and now the result follows from Lemma 3.3.

\section{Counterexample}
In this section we show that Theorem 3.1 can fail if the amenability hypothesis is dropped. First, in Section 4.1, we give a general construction, which obtains from a structure $M$ a new structure $M_{\mathrm{cyc}}$ in which algebraic formulas from $M$ give rise to forking formulas. Then in Section 4.2 we make some basic observations about the model theory of the binary tree $M=(2^{<\omega},\leqslant)$, and then in Section 4.3 we show that the structure $M_{\mathrm{cyc}}$ associated to it has the desired behavior.
\subsection{General Construction}
Let $L$ be a first-order language, and assume for notational convenience that it is one-sorted. We define a new language $L_{\mathrm{cyc}}$, which has two sorts $E$ and $B$, a function symbol $\pi:E\to B$, a ternary relation symbol $\mathrm{cyc}$ on the sort $E$, and all the symbols of $L$, taken as symbols on the sort $B$. Now, given an $L$-structure $M$, we define a corresponding $L_{\mathrm{cyc}}$-structure $M_{\mathrm{cyc}}$ by taking $E(M_{\mathrm{cyc}})=S^1\times M$ and $B(M_{\mathrm{cyc}})=M$, taking $\pi:S^1\times M\to M$ to be the projection map, and taking $\mathrm{cyc}$ to be the circular order on each fiber of $\pi$ and to induce no other relations. (So that $\mathrm{cyc}(x,y,z)\to \pi(x)=\pi(y)=\pi(z)$.)

Note that $M_{\mathrm{cyc}}$ is interpretable in the two-sorted structure $(S^1,M)$; in particular, if $M$ is NIP, then $M_{\mathrm{cyc}}$ will be too. Now, let $\mathfrak{C}$ be a saturated model of $\mathrm{Th}(M_{\mathrm{cyc}})$. By uniqueness of saturated models, $\mathfrak{C}$ can be identified with the projection $\bar{S}^1\times\bar{M}\to\bar{M}$, where $\bar{S}^1,\bar{M}$ are saturated models of the circular order and of $\mathrm{Th}(M)$, respectively. Note that the automorphisms of $\mathfrak{C}$ are in bijection with pairs $(\alpha,(f_e)_{e\in\bar{M}})$, where $\alpha$ is an automorphism of $\bar{M}$ and $f_e$ is an automorphism of $\bar{S}^1$ for each $e\in\bar{M}$; the associated automorphism of $\mathfrak{C}$ acts on the $E$-sort by taking $(u,e)$ to $(f_e(u),\alpha(e))$.

In the section we will focus on $L_{\mathrm{cyc}}(\bar{M})$-formulas of the form $\phi(\pi(x),b)$, where $\phi(z,b)$ is an $L(\bar{M})$-formula. We will show that $\phi(\pi(x),b)$ forks in $\mathfrak{C}$ if and only if $\phi(z,b)$ implies a disjunction of an algebraic $L$-formula and a forking $L(\bar{M})$-formula; first we give the backwards direction:

\begin{lemma}
    Let $\theta(z)$ be an algebraic $L$-formula and $\psi(z,d)$ an $L(\bar{M})$-formula that forks in $\bar{M}$; then the formula $\theta(\pi(x))\vee\psi(\pi(x),d)$ forks in $\mathfrak{C}$.
\end{lemma}
\begin{proof}
    It suffices to show that $\theta(\pi(x))$ and $\psi(\pi(x),d)$ each fork in $\mathfrak{C}$. To see that $\theta(\pi(x))$ forks in $\mathfrak{C}$, let $b_1,\dots,b_n$ be the realizations of $\theta(z)$; now $\theta(\pi(x))$ implies the disjunction $ \bigvee_{i\in[n]}(\pi(x)=b_i)$. But each formula $\pi(x)=b_i$ forks in $\mathfrak{C}$, since the induced structure on the fiber $\bar{S}^1\times\{b_i\}$ is just the circular order, and any formula in the circular order forks. So we need only show that $\psi(\pi(x),d)$ forks in $\mathfrak{C}$.
    
    By hypothesis, $\psi(z,d)$ implies a finite disjunction $\bigvee_{i\in[n]}\xi_i(z,e)$, where each $\xi_i(z,e)$ is an $L(\bar{M})$-formula that divides in $\bar{M}$, and so it suffices to show that each $\xi_i(\pi(x),e)$ divides in $\mathfrak{C}$. But if $(e_j)_{j\in\omega}$ is an indiscernible sequence from $\bar{M}$ with $e_0=e$ and $(\xi_i(z,e_j):j\in\omega)$ inconsistent, then $(\xi_i(\pi(x),e_j):j\in\omega)$ is also inconsistent, and $(e_j)_{j\in\omega}$ is still an indiscernible sequence in $\mathfrak{C}$, so that indeed $\xi_i(\pi(x),e)$ divides in $\mathfrak{C}$.
\end{proof}

In the rest of the section we will show the forwards direction of the claim above; first we need some auxiliary observations about types in $\mathfrak{C}$.

\begin{lemma}
    Let $(c_1,\dots,c_n)$ and $(c'_1,\dots,c'_n)$ be tuples from $E(\mathfrak{C})=\bar{S}^1\times\bar{M}$ and $a,a'$ elements of $E(\mathfrak{C})$ such that (i) $\pi(a)\neq\pi(c_i)$ and $\pi(a')\neq\pi(c'_i)$ for each $i$, (ii) $(\pi(a),\pi(\bar{c}))$ and $(\pi(a'),\pi(\bar{c}'))$ have the same type in $\bar{M}$, and (iii) $\bar{c}$ and $\bar{c}'$ have the same type in $\mathfrak{C}$. Then $(a,\bar{c})$ and $(a',\bar{c}')$ have the same type in $\mathfrak{C}$.
\end{lemma}
\begin{proof}
    By (iii), let $(\alpha,(f_e)_{e\in\bar{M}})$ be the data associated to an automorphism of $\mathfrak{C}$ taking $\bar{c}$ to $\bar{c}'$. Also, by (ii), let $\beta$ be an automorphism of $\bar{M}$ taking $(\pi(a),\pi(\bar{c}))$ to $(\pi(a'),\pi(\bar{c}'))$. Finally, using that any two elements of $\bar{S}^1$ have the same type, let $f$ be an automorphism of $\bar{S}^1$ taking the first coordinate of $a$ to the first coordinate of $a'$. Now, for $e\in\bar{M}$, define an automorphism $g_e$ of $\bar{S}^1$ by cases:
    \[
      g_e =
      \begin{cases}
        f & \text{if }e=\pi(a) \\
        f_e & \text{if }e=\pi(c_i)\text{ for some }i\\
        \mathrm{id}_{\bar{S}^1} & \text{otherwise}.
      \end{cases}
    \] By condition (i), this is well-defined, and the automorphism of $\mathfrak{C}$ associated to $(\beta,(g_e)_{e\in\bar{M}})$ then takes $(a,\bar{c})$ to $(a',\bar{c}')$, as needed.
\end{proof}

\begin{corollary}
    Suppose $C\subseteq E(\mathfrak{C})=\bar{S}^1\times\bar{M}$ and $a\in E(\mathfrak{C})$ are such that (i) $\pi(a)\notin\pi(C)$, and (ii) $\tp(\pi(a)/\pi(C))$ does not divide in $\bar{M}$. Then $\tp(a/C)$ does not divide in $\mathfrak{C}$.
\end{corollary}
\begin{proof}
    Let $(\bar{c}_i)_{i\in\omega}$ be an indiscernible sequence, where $\bar{c}_0$ is a finite tuple from $C$. We wish to find $a'$ such that $(a',\bar{c}_i)$ and $(a,\bar{c}_0)$ have the same type in $\mathfrak{C}$ for every $i\in\omega$. Since $\tp(\pi(a)/\pi(C))$ does not divide, we may (picking a first coordinate arbitrarily) find some $a'\in E(\mathfrak{C})$ such that $(\pi(a'),\pi(\bar{c}_i))$ and $(\pi(a),\pi(\bar{c}_0))$ have the same type in $\bar{M}$ for every $i\in\omega$.  But, since $\pi(a)\notin\pi(C)$, this gives the desired result by Lemma 4.2.
\end{proof}

Now we can give the desired characterization.

\begin{lemma} Let $\phi(z,b)$ be an $L({\bar{M}})$-formula. Then $\phi(\pi(x),b)$ forks in $\mathfrak{C}$ if and only if there is an $L$-formula $\theta(z)$ and an $L(\bar{M})$-formula $\psi(z,d)$ such that $\theta(z)$ is algebraic, $\psi(z,d)$ forks in $\bar{M}$, and $\phi(z,b)\vdash\theta(z)\vee\psi(z,d)$.
\end{lemma}
\begin{proof}
    The backwards direction is Lemma 4.1, so we show the forwards direction. Suppose for contradiction that $\phi(\pi(x),b)$ forks in $\mathfrak{C}$ but that the desired conclusion does not hold. Then there are some finite $C\subset E(\mathfrak{C})$ and $D\subset B(\mathfrak{C})=\bar{M}$ such that, for any $a\in \bar{S}^1\times\bar{M}$ realizing $\phi(\pi(x),b)$, $\tp(a/C,D)$ divides. Expanding $C$ if necessary, we may assume without loss of generality that $\pi(C)\supseteq D$, so that $\tp(a/C)$ divides for any $a\models\phi(\pi(x),b))$.
    
    Now, by hypothesis, $\phi(z,b)$ is consistent with $\neg\theta(z)\wedge\neg\psi(z,d)$ for every algebraic $L$-formula $\theta(z)$ and every $L(\pi(C))$-formula $\psi(z,d)$ that forks in $\bar{M}$. Formulas of these latter two kinds form an ideal, so by compactness (and picking an arbitrary first coordinate) we may find some $a\in E(\mathfrak{C})$ such that:
    \begin{enumerate}
        \item $\pi(a)\models\phi(z,b)$.
        \item $\tp(\pi(a)/\varnothing)$ is non-algebraic.
        \item $\tp(\pi(a)/\pi(C))$ does not fork in $\bar{M}$.
    \end{enumerate} (2) and (3) together imply that $\pi(a)\notin\pi(C)$. By (3) and Corollary 4.3, this implies that $\tp(a/C)$ does not divide in $\mathfrak{C}$. By (1), this contradicts the choice of $C$.
\end{proof}

\subsection{Binary Tree}
Let $L=\{\leqslant\}$ be the language of trees and let $M$ be the tree $(2^{<\omega},\leqslant)$. Note that $M$ is NIP, and indeed dp-minimal, since any tree is; see \citep{simon_trees}. Let $\bar{M}$ be a saturated model of $\mathrm{Th}(M)$. Note that $M$ is a meet-tree, and so $\bar{M}$ is as well. We have $\acl(\varnothing)=M$, so that $M$ is the unique minimal elementary substructure of $\bar{M}$. By a back-and-forth argument, any two elements $b,b'\in\bar{M}\setminus M$ have the same type over $\varnothing$, and the non-algebraic $1$-types over $M$ correspond to paths through $2^{<\omega}$: given $f\in 2^{\omega}$, there is a unique $1$-type over $M$ determined by $\{x> f|_n:n\in\omega\}$, and conversely any unrealized $1$-type over $M$ is of that form.

Now, suppose $I=(b_i)_{i\in\omega}$ is any non-constant indiscernible sequence of elements of $\bar{M}$; then, since $M=\acl(\varnothing)$, $I$ is indiscernible over $M$, and so all the elements of $I$ determine the same path through $M$. In particular, for every $n\in\omega$, we have $b_1\wedge\dots\wedge b_n\notin M$.

\begin{lemma}
    Let $b\in\bar{M}\setminus M$. Then the formula $\phi(z,b)\equiv z\leqslant b$ does not imply a disjunction of an algebraic formula without parameters and an $L(\bar{M})$-formula that forks.
\end{lemma}
\begin{proof}
    It suffices to show that $\phi(z,b)\wedge\neg\theta(z)$ does not fork over $\varnothing$ for every algebraic $\theta(z)$ without parameters. So suppose otherwise; then $\phi(z,b)\wedge\neg\theta(z)$ implies a finite disjunction of formulas that each divide over $\varnothing$. Since $M=\acl(\varnothing)$, a formula divides over $\varnothing$ if and only if it divides over $M$, and hence $\phi(z,b)\wedge\neg\theta(z)$ implies a finite disjunction of formulas that each divide over $M$; in other words $\phi(z,b)\wedge\neg\theta(z)$ forks over $M$. But $\mathrm{Th}(M)$ is NIP, so forking and dividing coincide over models, and hence $\phi(z,b)\wedge\neg\theta(z)$ divides over $M$. Let $(b_i)_{i\in\omega}$ be an indiscernible sequence witnessing this; then for some $n\in\omega$ we have $\phi(z,b_1)\wedge\dots\wedge\phi(z,b_n)\vdash\theta(z)$. In particular, we have $b_1\wedge\dots\wedge b_n\models\theta(z)$. But $\theta$ is algebraic, so this forces $b_1\wedge\dots\wedge b_n\in M$, contradicting the remark above this lemma.
\end{proof}

\subsection{Main example} Now we can give the main example. Let $M$ be the tree from Section 2. Let $\mathfrak{C}$ be a saturated model of $\mathrm{Th}(M_{\mathrm{cyc}})$, with $E$-sort $\bar{S}^1\times\bar{M}$. Pick any $b\in\bar{M}\setminus M$, and let $\phi(z,b)\equiv z\leqslant b$.

\begin{lemma}
    The formula $\phi(\pi(x),b)$ does not fork.
\end{lemma}
\begin{proof}
    This follows from Lemma 4.4 and Lemma 4.5.
\end{proof}

\begin{lemma}
    Let $\mu$ be a automorphism-invariant Keisler measure on the $E$-sort of $\bar{M}$. Then $\mu(\phi(\pi(x),b))=0$.
\end{lemma}
\begin{proof}
    Suppose otherwise that $\mu(\phi(\pi(x),b))=\varepsilon>0$. Recalling that any two elements of $\bar{M}\setminus M$ have the same type over $\varnothing$, we thus have $\mu(\phi(\pi(x),b'))=\varepsilon$ for every $b'\in\bar{M}\setminus M$. Let $n$ be such that $n\varepsilon>1$, and let $m$ be any natural number with $2^m>n$; I claim that $\mu\left(P_m(\pi(x))\right)>0$, where $P_m(z)\equiv\exists^{\leqslant m}u(u\leqslant z)$ is the formula of $\bar{M}$ expressing that $z$ is of height at most $m$.
    
    To see this, suppose $\mu\left(P_m(\pi(x))\right)=0$. Since $2^m>n$, we may find \textit{distinct} elements $c_1,\dots,c_n\in M$ of height precisely $m$. Pick any $b_i\in\bar{M}\setminus M$ with $b_i\geqslant c_i$; then $b_i\wedge b_j\models P_m(z)$ for each $i\neq j$. Hence $\phi(z,b_i)\wedge\phi(z,b_j)\vdash P_m(z)$ for each $i\neq j$; since $P_m(\pi(x))$ has measure $0$ under $\mu$, also $$\mu\big(\phi(\pi(x),b_i)\wedge\phi(\pi(x),b_j)\big)=0$$ for each $i\neq j$. But then $$\mu\left(\bigvee_{i\in[n]}\phi(\pi(x),b_i)\right)=\sum_{i\in[n]}\mu(\phi(\pi(x),b_i))=n\varepsilon>1,$$ a contradiction.

    So indeed $\mu(P_m(\pi(x)))>0$. But $P_m(z)$ is algebraic, so the formula $P_m(\pi(x))$ forks over $\varnothing$ by Lemma 4.1, contradicting that $\mu$ is $\varnothing$-invariant.
\end{proof}

Thus $\mathrm{Th}(M_{\mathrm{cyc}})$ is an NIP theory for which, in a saturated model, there are non-forking formulas given measure $0$ by any automorphism-invariant Keisler measure; this gives the desired example.

\newpage


\end{document}